\documentclass[12 pt, a4 paper, twoside]{article}
\usepackage{amsmath,amssymb}
\usepackage{longtable,mathrsfs}
\usepackage{hyperref}
\usepackage{amsmath, amsthm}
\usepackage[mathscr]{euscript}
\newtheorem{thm}{Theorem}

\newtheorem{lem}{Lemma}
\newtheorem{cor}{Corollary}

\newtheorem{re}{Remark}
\newtheorem{defn}{Definition}
\newcommand{\B}{\mathbb{B}}
\newcommand{\R}{\mathbb{R}}
\newcommand{\N}{\mathbb{N}}
\newcommand{\I}{\mathfrak{I}}
\newcommand{\D}{\mathfrak{D}}
{ \setlength{\oddsidemargin}{-2mm}
  \setlength{\evensidemargin}{6mm} }
  \newlength{\titleright}
\allowdisplaybreaks
\begin{document}
\title{\vspace*{-6pc}{\bf On Ulam type stability for nonlinear implicit fractional differential equations}}

\vspace{1cm}\author{ {\small D. B. Dhaigude$^{1}$
      \footnote{Corresponding author. Email address: sandeeppb7@gmail@.com,  Tel:+91--9421475347}, Sandeep P. Bhairat$^{2}$}\ \\
{\footnotesize \it $^{1,2}$Department of mathematics, Dr. B. A. M. University, Aurangabad--431 004 (M.S.) India.}\\}

\date{}

\maketitle

{\vspace*{0.5pc} \hrule\hrule\vspace*{2pc}}

\hspace{-0.8cm}{\bf Abstract}\\
In present paper, we establish sufficient conditions for existence and stability of solutions for system of nonlinear implicit fractional differential equations. The main techniques are based on method of successive approximations. Finally, an illustrative example is given to show the applicability of our theoretical results.\\

\thispagestyle{empty}
\section{Introduction}
The year 1695, a letter conversation of Leibniz and L'Hospital, treated as birth of fractional calculus but the first precise definition of fractional derivative and primitive was introduced at the end of ninteenth century by Liouville and Riemann. This calculus of arbitrary order initially appeared  as a theoretical development in the mathematical analysis, but in the past few decades, it is proved to be an excellent tool in the description of many processes occurring naturally. Numerous applications are found in nonlinear waves of earthquake, modelling the seepage flow in porous media and in fluid dynamics, memory mechanism and hereditary properties of materials. Some recent existence-uniqueness results of solutions for fractional differential equations with initial as well as boundary conditions can be found in \cite{as}-\cite{b1,db1} and see the books by Kilbas et.al. \cite{kst}, Podlubny \cite{pi}, Samko et.al. \cite{skm} and references therein.

In 1940, Ulam \cite{u1} proposed a general Ulam stability problem in the talk before the Mathematics Club of University of Wisconsin in which he discussed a number of important unsolved problems: {\it When is it true that by slightly changing the hypothesis of the theorem one can still assert that the thesis of the theorem remains true or approximately true ?}, also see \cite{u2}. In the following year, Hyers \cite{dh} affirmatively answered partially to the Ulams' question. Further in 1978, Rassias \cite{r1}-\cite{r3} generalized the results of Hyers' and since then the stability of functional equations have been investigated by many researchers as an emerging field of mathematical analysis. As a consequence of Hyers' result, a great number of papers on the subject have been appeared, extending and generalizing Ulams' problem and Hyers' theorem in various aspects, \cite{b2,dh}-\cite{sm4,jw1,jw2}.

The stability analysis of all kinds of equations have attracted the concentration of many researchers in the books \cite{cc,sm5,kil,r2,r3,u1,u2}. S-M. Jung and his co-authors obtained numerous results on the Ulam-type stability of linear as well as nonlinear differential equations. S-M. Jung derived the Hyers-Ulam stability of first order linear both ordinary and partial differential equations in the series of papers \cite{sm1}-\cite{sm4}. The classical concept of Ulam-Hyeres stability has applicable significance since it means that if we are dealing with Ulam-Hyers stable system then one does not seek the exact solution. All what is required is to find a function which satisfies a suitable approximation inequation. This approach is quite useful in many applications such as numerical analysis and optimization, where seeking the exact solution is impossible.

Although, the local stability and Mittag-Leffler stability results are obtained in the literature by fixed point theory \cite{b2,db2,dz,jw2,wl}, to the best of our knowledge, there are very rare works on the Ulam stability for nonlinear implicit fractional differential equations by using method of successive approximations \cite{ch,db3,db4,kkss}. Integral inequalities plays a major role in the analysis of various functional equations and their qualitative properties. For references see \cite{b2,cc,cs,sm1}-\cite{kil,dq,skm},\cite{jw1}-\cite{wl}. Following the idea of \cite{kkss}, in this paper, we will study the existence, uniqueness and four types of Ulam-type stabilities for the following nonlinear initial value problem (IVP) of system of implicit fractional differentia equations
\begin{equation}\label{a}\begin{cases}
\D_{1}^{\alpha}x(t)=f(t,x(t),{\D_{1}^{\alpha}x(t)}),\quad t\in{J}=[1,T],T>1,\\
  x^{(k)}(1)=x_k,\quad x_k\in{\R}^{n},\quad k=0,1,\cdots,m-1,
  \end{cases}
\end{equation}
for some $\alpha\in(m-1,m],m\in\N,$ where $f:J\times{\R}^{n}\times{\R}^{n}\to{\R}^{n}$ be a nonlinear continuous function, $x:J\to{\R}^{n}$ and $\D_{1}^{\alpha}$ denotes the Caputo-Hadamard derivative of order $\alpha.$

The layout of this work is as follow: In section 2, some notations and preparation results are given. Section 3, the main results, divided into three subsections from which first is devoted to existence and uniqueness of solution for considered IVP \eqref{a} followed by four types of Ulam-type stabilities and conclude the section with $E_{\alpha}-$Ulam type stability results. In last section, we give the counter example in support of the theory.
\section{Preliminaries}
Consider the real space ${\R}^{n}$ with the norm $||\cdot||$ and denote by $\B=C^m(J,{\R}^{n})-$the Banach space of all continuous functions from $J$ into ${\R}^{n}$ having $m^{th}$ order derivatives endowed with supremum norm ${||\cdot||}_{\B}.$ The special function frequently used in the solution of fractional differential equations is the Mittag-Leffler function denoted as $E_{\alpha}(z)$ and defined by
\begin{equation}\label{ml}
E_{\alpha}(z)=\sum_{k=0}^{\infty}\frac{z^k}{\Gamma(\alpha k+1)},\quad z\in\R,Re(\alpha)>0,
\end{equation}
where $\Gamma(x)=\int_{0}^{\infty}e^{-t}t^{x-1}dt,\,x>0,$ is the Gamma function.

\begin{defn}\cite{kst}
The Hadamard fractional integral of order $\alpha>0$ for a continuous function $g(t):[1,+\infty)\to\R$ is defined as
\begin{equation}\label{hi} \I_{1}^{\alpha}g(t)=\frac{1}{\Gamma(\alpha)}\int_{1}^{t}\bigg(\log{\frac{t}{s}}\bigg)^{\alpha-1}g(s)\frac{ds}{s},\quad \alpha>0.
\end{equation}
\end{defn}
\begin{defn}\cite{kst}
The Caputo-Hadamard fractional derivative of order $\alpha$ for a continuous function $g(t):[1,+\infty)\to\R$ is defined as
\begin{equation}\label{ch}
\D_{1}^{\alpha}g(t)=\frac{1}{\Gamma(n-\alpha)}\int_{1}^{t}\bigg(\log{\frac{t}{s}}\bigg)^{n-\alpha-1}\delta^{n}( g)(s)\frac{ds}{s},\quad n-1<\alpha<n,
\end{equation}
where $\delta^n=\bigg(t\frac{d}{dt}\bigg)^{n},n\in\N.$
\end{defn}
\begin{lem}\cite{kst}
Let $m-1<\alpha\leq m,m\in\N$ and $g\in C^m[1,T].$ Then
\begin{equation*}
\I_{1}^{\alpha}[\D_{1}^{\alpha}g(t)]=g(t)-\sum_{k=0}^{m-1}\frac{g^{(k)}(1)}{\Gamma(k+1)}(\log{t})^k.
\end{equation*}
\end{lem}
\begin{lem}\cite{kst}
For all $\mu>0$ and $\nu>-1,$
\begin{equation*}
\frac{1}{\Gamma(\mu)}\int_{1}^{t}\bigg(\log{\frac{t}{s}}\bigg)^{\mu-1}(\log{s})^{\nu}\frac{ds}{s}=\frac{\Gamma(\nu+1)}{\Gamma(\mu+\nu+1)}(\log{t})^{\mu+\nu}.
\end{equation*}
\end{lem}
\begin{lem}\cite{kst}
Let $g(t)=t^{\mu},$ where $\mu\geq0$ and let $m-1<\alpha\leq m,m\in\N.$ Then
\begin{equation*}
\D_{1}^{\alpha}(\log{t})^{\mu}=\begin{cases}
0, & \mbox{if}\,\,\mu\in\{0,1,\cdots,m-1\},\\
\frac{\Gamma(\mu+1)}{\Gamma(\mu+\alpha+1)}(\log{t})^{\mu+\alpha}, &\mbox{if}\,\, \mu\in\N,\mu\geq m\,\,\mbox{or}\,\,\mu\notin\N,\mu>m-1.
\end{cases}
\end{equation*}
\end{lem}
\begin{lem}\cite{dz}
For any $t\in[1,T],$
\begin{equation*}
u(t)\leq a(t)+b(t)\int_{1}^{t}\bigg(\log{\frac{t}{s}}\bigg)^{\alpha-1}u(s)\frac{ds}{s},
\end{equation*}
where all the functions are not negative and continuous. The constant $\alpha>0,b$ is a bounded and monotonic increasing function on $[1,T),$ then,
\begin{equation*}
u(t)\leq a(t)+\int_{1}^{t}\bigg[\sum_{n=1}^{\infty}\frac{(b(t)\Gamma(\alpha))^{n}}{\Gamma(n\alpha)}\bigg(\log{\frac{t}{s}}\bigg)^{n\alpha-1}a(s)\bigg]\frac{ds}{s},\quad t\in[1,T).
\end{equation*}
\end{lem}
\begin{re}
Under the hypothesis of Lemma 4, let $a(t)$ be a nondecreasing function on $[1,T).$ Then
\begin{equation*}
u(t)\leq a(t)E_{\alpha}\big(b(t)\Gamma(\alpha){\log{t}}^{\alpha}\big).
\end{equation*}
\end{re}
\begin{lem}\cite{kkss}
Let $X$ be a Banach space and let $D$ be an operator which maps the element of $X$ into itself for which $D^r$ is a contraction, where $r$ is a positive integer then $D$ has a unique fixed point.
\end{lem}
\begin{defn}
A function $x\in\B$ is said to be a solution of problem \eqref{a} if $x$ satisfies nonlinear implicit fractional differential system of equations $ \D_{1}^{\alpha}x(t)=f(t,x(t),{\D_{1}^{\alpha}x(t)})$ on $J$ together with initial conditions $x^{(k)}(1)=x_k,\, k=0,1,\cdots,m-1,\,\,x_k\in{\R}^{n}$, $m-1<\alpha\leq m,\,\, m\in\N.$
\end{defn}
\section{Main results}
In this section, we present our main results concerning the existence and stability of solutions for fractional differential system \eqref{a}.
\subsection{Existence-Uniqueness}
We need the following lemma to establish the existence of solution which is equivalence between IVP \eqref{a} and equivalent fractional nonlinear integral equation
\begin{equation}\label{b}
x(t)=\sum_{k=0}^{m-1}\frac{x_k}{\Gamma(k+1)}(\log{t})^k+\frac{1}{\Gamma(\alpha)}\int_{1}^{t}\bigg(\log{\frac{t}{s}}\bigg)^{\alpha-1}p(s)\frac{ds}{s},\quad t\in{J},
\end{equation}
where $p\in\B$ satisfies the functional equation
\begin{equation}\label{c}
p(t)=f\bigg(t,\sum_{k=0}^{m-1}\frac{x_k}{\Gamma(k+1)}(\log{t})^k+\I_{1}^{\alpha}p(t),p(t)\bigg),\quad t\in{J}.
\end{equation}
\begin{lem}
Let $f:J\times{\R}^{n}\times{\R}^{n}\to{\R}^{n}$ be a continuous function. Then system \eqref{a} is equivalent to the fractional integral equation \eqref{b}.
\end{lem}
\begin{proof}
First we prove the necessity. Let $x:J\to{\R}^{n}$ in $\B$ is a solution of IVP \eqref{a}. Denote
\begin{equation}\label{d}
\D_{1}^{\alpha}x(t)=p(t),\quad t\in{J},m-1<\alpha\leq m,m\in\N,
\end{equation}
where $p\in\B.$ Applying $\I_{1}^{\alpha}$ on both sides of \eqref{d} and in the light of Lemma 3, we obtain
\begin{align*}
\I_{1}^{\alpha}{\D_{1}^{\alpha}}x(t)&= \I_{1}^{\alpha}p(t)\\
x(t)-\sum_{k=0}^{m-1}\frac{x^{(k)}(1)}{\Gamma(k+1)}&(\log{t})^k= \I_{1}^{\alpha}p(t)\\
x(t)=\sum_{k=0}^{m-1}\frac{x_{k}}{\Gamma(k+1)}&(\log{t})^k+ \I_{1}^{\alpha}p(t),\quad t\in{J}.
\end{align*}
Thus, differential equation of IVP \eqref{a} becomes
\begin{align*}
p(t)&=f(t,x(t),p(t))\\
&=f\bigg(t,\sum_{k=0}^{m-1}\frac{x_k}{\Gamma(k+1)}(\log{t})^k+\I_{1}^{\alpha}p(t),p(t)\bigg),\quad t\in{J}.
\end{align*}
This proves $x$ is a solution of fractional integral equation \eqref{b}, where $p\in\B$ satisfies \eqref{c}.

Now we prove the sufficiency. Let $x:J\to{\R}^{n}$ in $\B$ satisfies the fractional integral equation \eqref{b}, where $p\in\B$ satisfies functional equation \eqref{c}. Then equation \eqref{b} can be written as
\begin{equation}\label{e}
x(t)=\sum_{k=0}^{m-1}\frac{x_k}{\Gamma(k+1)}(\log{t})^k+ \I_{1}^{\alpha}p(t),\quad t\in{J}.
\end{equation}
Since $p$ is continuous then using Lemma 3 and applying $\D_{1}^{\alpha}$ on both sides of \eqref{e}, we obtain
\begin{align}\label{f}
\D_{1}^{\alpha}x(t)&=\sum_{k=0}^{m-1}\frac{x_k}{\Gamma(k+1)}\D_{1}^{\alpha}[(\log{t})^k]+\D_{1}^{\alpha}{\I_{1}^{\alpha}}p(t)\nonumber\\
&=p(t),\qquad t\in{J}.
\end{align}
But $p$ satisfies functional equation \eqref{c}. Thus using equations \eqref{e} and \eqref{f} in \eqref{c}, we obtain
\begin{equation*}
  \D_{1}^{\alpha}x(t)=f(t,x(t),\D_{1}^{\alpha}x(t)),\quad t\in{J}.
\end{equation*}
Further from integral equation \eqref{b},
\begin{equation*}
x^{(k)}(t)=x_k+ \I_{1}^{\alpha-k}p(t),\quad p\in\B, t\in{J}, k=0,1,\cdots,m-1.
\end{equation*}
One can easily verify that the initial condition of system \eqref{a} is satisfied by setting $t=1$ in above last equation. The proof is completed.
\end{proof}

Next, we introduce the following assumptions:
\begin{description}
\item[(H1)] A continuous function $f:J\times{\R}^{n}\times{\R}^{n}\to{\R}^{n}$ satisfies the Lipschitz-type condition: for $x,y,\tilde{x},\tilde{y}\in{\R}^{n}$ there exist constants $M>0$ and $0<N<1$ such that
  \begin{equation*}
  ||f(t,x,y)-f(t,\tilde{x},\tilde{y})||\leq M||x-\tilde{x}||+N||y-\tilde{y}||,\quad t\in{J}.
  \end{equation*}
\item (H2) Let $\Phi\in C({J},{\R}_+)$ be a nondecreasing function. There exists a constant $K>0$ satisfying $0<K\theta<1$ and
      \begin{equation*}
      \bigg{\|}\frac{1}{\Gamma(\alpha)}\int_{1}^{t}\bigg(\log{\frac{t}{s}}\bigg)^{\alpha-1}\Phi(s)\frac{ds}{s}\bigg{\|}\leq K\Phi(t),\quad t\in{J},
      \end{equation*}
      where $\theta=\frac{M}{1-N}>0.$
\end{description}
\begin{thm}
Suppose that a function $f$ satisfies assumption \text{(H1)}. Then the IVP \eqref{a} has a unique solution $x:J\to{\R}^{n}.$
\end{thm}
\begin{proof}
In accordance with Lemma 6, the IVP \eqref{a} can be written as a fixed point problem. For this, consider the operator $I:\B\to\B$ as follows.
\begin{equation*}
Ix(t)=\sum_{k=0}^{m-1}\frac{x_k}{\Gamma(k+1)}(\log{t})^k+\frac{1}{\Gamma(\alpha)}\int_{1}^{t}\bigg(\log{\frac{t}{s}}\bigg)^{\alpha-1}p(s)\frac{ds}{s},\quad t\in{J},
\end{equation*}
where $p\in\B$ satisfies the functional equation $p(t)=f(t,x(t),p(t))$ for $t\in{J}.$

It is sufficient to prove that operator $I$ has a fixed point. More precisely, using the principle of mathematical induction, we prove that for any $x,y\in\B$ and $t\in{J},$
\begin{equation}\label{g}
||I^{j}x(t)-I^{j}y(t)||\leq\frac{(\theta(\log{t})^{\alpha})^{j}}{\Gamma(\alpha j+1)}{||x-y||}_{\B},\quad j\in\N,
\end{equation}
where $\theta=\big(\frac{M}{1-N}\big)>0.$

For any $x,y\in\B$ and $t\in{J},$ by definition of operator $I,$ we have
\begin{equation}\label{h}
||Ix(t)-Iy(t)||\leq\frac{1}{\Gamma(\alpha)}\int_{1}^{t}\bigg(\log{\frac{t}{s}}\bigg)^{\alpha-1}||p(s)-q(s)||\frac{ds}{s},
\end{equation}
where $p,q\in\B$ satisfies the functional equations
\begin{equation*}
p(t)=f(t,x(t),p(t))\quad \text{and  }q(t)=f(t,y(t),q(t)),\quad t\in{J}.
\end{equation*}
For any $t\in{J},$ from assumption \text{(H1)} we have
\begin{align*}
||p(t)-q(t)||&=||f(t,x(t),p(t))-f(t,y(t),q(t))||\\
&\leq M||x(t)-y(t)||+N||p(t)-q(t)||\\
&=\theta||x(t)-y(t)||,\quad t\in{J}.
\end{align*}
Hence inequality \eqref{h} becomes
\begin{align*}
||Ix(t)-Iy(t)||&\leq \frac{\theta}{\Gamma(\alpha)}\int_{1}^{t}\bigg(\log{\frac{t}{s}}\bigg)^{\alpha-1}||x(s)-y(s)||\frac{ds}{s}\\
&\leq \frac{\theta}{\Gamma(\alpha)}\bigg(\int_{1}^{t}\bigg(\log{\frac{t}{s}}\bigg)^{\alpha-1}\frac{ds}{s}\bigg){||x-y||}_{\B}.
\end{align*}
Therefore
\begin{equation*}
||Ix(t)-Iy(t)||\leq \frac{\theta(\log{t})^{\alpha}}{\Gamma(\alpha+1)}{||x-y||}_{\B},\quad t\in{J}.
\end{equation*}
Thus the inequality \eqref{g} is true for $j=1.$ Let us now assume that it hold for $j=r,r\in\N:$
\begin{equation}\label{i}
||I^rx(t)-I^ry(t)||\leq \frac{[\theta(\log{t})^{\alpha}]^r}{\Gamma(r\alpha+1)}{||x-y||}_{\B},\quad t\in{J}.
\end{equation}
We prove that inequality \eqref{g} holds for $j=r+1.$ By definition of operator $I,$
\begin{align*}
||I^{r+1}x(t)-I^{r+1}y(t)||&=||I(I^rx(t))-I(I^ry(t))||\\
&\leq\frac{\theta}{\Gamma(\alpha)}\int_{1}^{t}\bigg(\log{\frac{t}{s}}\bigg)^{\alpha-1}||h(s)-g(s)||\frac{ds}{s},\quad t\in{J},
\end{align*}
where $h,g\in\B$ are such that
\begin{equation*}
h(t)=f(t,I^rx(t),h(t))\,\,\text{and}\,\, g(t)=f(t,I^ry(t),g(t)),\quad t\in{J}.
\end{equation*}
By assumption \text{(H1)}, we have
\begin{equation*}
||h(t)-g(t)||\leq\theta ||I^rx(t)-I^ry(t)||,\quad t\in{J}.
\end{equation*}
Hence we obtain
\begin{equation*}
||I^{r+1}x(t)-I^{r+1}y(t)||\leq \frac{\theta}{\Gamma(\alpha)}\int_{1}^{t}\bigg(\log{\frac{t}{s}}\bigg)^{\alpha-1}||I^rx(s)-I^ry(s)||\frac{ds}{s},\quad t\in{J}.
\end{equation*}
By using inequality \eqref{i} and Lemma 2, the above inequality takes the form
\begin{align*}
||I^{r+1}x(t)-I^{r+1}y(t)||&\leq \frac{\theta}{\Gamma(\alpha)}\int_{1}^{t}\bigg(\log{\frac{t}{s}}\bigg)^{\alpha-1}
\frac{[\theta(\log{s})^{\alpha}]^r}{\Gamma(r\alpha+1)}{||x-y||}_{\B} \frac{ds}{s}\\
&\leq \frac{\theta^{r+1}}{\Gamma(r\alpha+1)}\bigg(\frac{1}{\Gamma(\alpha)}\int_{1}^{t}\bigg(\log{\frac{t}{s}}\bigg)^{\alpha-1}
{(\log{s})^{r\alpha}}\frac{ds}{s}\bigg){||x-y||}_{\B}\\
&=\frac{\theta^{r+1}}{\Gamma(r\alpha+1)}\bigg[\frac{\Gamma(r\alpha+1)}{\Gamma(r\alpha+\alpha+1)}(\log{t})^{r\alpha+\alpha}\bigg]
{||x-y||}_{\B},\quad t\in{J}.
\end{align*}
Therefore
\begin{equation*}
||I^{r+1}x(t)-I^{r+1}y(t)||\leq \frac{[\theta(\log{t})^{\alpha}]^{r+1}}{\Gamma((r+1)\alpha+1)}{||x-y||}_{\B},\quad t\in{J}.
\end{equation*}
This proves the inequality \eqref{g} is true for $j=r+1.$ The principle of mathematical induction completes the proof of inequality \eqref{g}.

Again from the inequality \eqref{g}, we obtain
\begin{align*}
||I^{j}x-I^{j}y||=\sup_{t\in{J}}||I^{j}x(t)-I^{j}y(t)||\leq \frac{[\theta(\log{T})^{\alpha}]^{j}}{\Gamma(j\alpha+1)}{||x-y||}_{\B}.
\end{align*}
By definition of Mittag-Leffler funciton \eqref{ml},
\begin{equation*}
E_{\alpha}(\theta(\log{T})^{\alpha})=\sum_{j=0}^{\infty}\frac{(\theta(\log{T})^{\alpha})^j}{\Gamma(j\alpha+1)}.
\end{equation*}
Note that $\frac{[\theta(\log{t})^{\alpha}]^{j}}{\Gamma(j\alpha+1)}$ is the $j^{th}$ term of the convergent series of nonnegative real numbers. This gives
\begin{equation*}
\lim_{j\to\infty}\frac{[\theta(\log{T})^{\alpha}]^{j}}{\Gamma(j\alpha+1)}=0.
\end{equation*}
The choice of $j\in\N$ for $\frac{[\theta(\log{T})^{\alpha}]^{j}}{\Gamma(j\alpha+1)}<1$ shows that $I^j$ is a contraction map. Thus by modified version of contraction mapping principle, Lemma 5, $I$ has a unique fixed point $x:{J}\to{\R}^{n}$ in $\B,$ which is the unique solution of IVP \eqref{a}. This complete the proof.
\end{proof}

\subsection{Ulam-type Stability}
In this subsection, we obtain the sufficient conditions for stability of solution for IVP \eqref{a}. Particularly, four types of Ulam-type stability results are discussed. i.e. Ulam-Hyers stability, generalized Ulam-Hyers stability, Ulam-Hyers-Rassias stability and generalized Ulam-Hyers-Rassias stability. Also note that, Ulam-type stabilities are a special types of data dependance of solutions.

First we consider the Ulam-type stability results for problem \eqref{a}. Let $\epsilon>0$ and $\Phi:{J}\to{\R}_{+}$ be a continuous function. We consider the following inequations:
\begin{equation}\label{j}
||\D_{1}^{\alpha}y(t)-f(t,y(t),{\D_{1}^{\alpha}y(t)})||\leq\epsilon,\qquad t\in{J},
\end{equation}
\begin{equation}\label{k}
||\D_{1}^{\alpha}y(t)-f(t,y(t),{\D_{1}^{\alpha}y(t)})||\leq\Phi(t),\quad t\in{J},
\end{equation}
\begin{equation}\label{l}
||\D_{1}^{\alpha}y(t)-f(t,y(t),{\D_{1}^{\alpha}y(t)})||\leq\epsilon\Phi(t),\quad t\in{J}.
\end{equation}
\begin{defn}
The problem \eqref{a} is Ulam-Hyers stable if there exists a real number $K_f>0$ such that for each $\epsilon>0$ and for each solution $y:{J}\to{\R}^{n}$ in $\B$ of inequality \eqref{j}, there exists a solution $x:{J}\to{\R}^{n}$ of problem \eqref{a} in $\B$ with
\begin{equation*}
||y(t)-x(t)||\leq\epsilon K_f;\qquad t\in{J}.
\end{equation*}
\end{defn}
\begin{defn}
The problem \eqref{a} is generalized Ulam-Hyers stable if there exists $\psi\in C({\R}_{+},{\R}_{+}), \psi(0)=0$ such that for each $\epsilon>0$ and for each solution $y:{J}\to{\R}^{n}$ in $\B$ of inequality \eqref{j}, there exists a solution $x:{J}\to{\R}^{n}$ of problem \eqref{a} in $\B$ with
\begin{equation*}
||y(t)-x(t)||\leq \psi(\epsilon);\qquad t\in{J}.
\end{equation*}
\end{defn}
\begin{defn}
The problem \eqref{a} is Ulam-Hyers-Rassias stable with respect to $\Phi,$ if there exists a real number $K_{f,\phi}>0$ such that for each $\epsilon>0$ and for each solution $y:{J}\to{\R}^{n}$ in $\B$ of inequality \eqref{l}, there exists a solution $x:{J}\to{{\R}^{n}}$ of problem \eqref{a} in $\B$ with
\begin{equation*}
||y(t)-x(t)||\leq\epsilon K_{f,\phi}\Phi(t);\qquad t\in{J}.
\end{equation*}
\end{defn}
\begin{defn}
The problem \eqref{a} is generalized Ulam-Hyers-Rassias stable with respect to $\Phi,$ if there exists a real number $K_{f,\phi}>0$ such that for each $\epsilon>0$ and for each solution $y:{J}\to{\R}^{n}$ in $\B$ of inequality \eqref{k}, there exists a solution $x:{J}\to{\R}^{n}$ of problem \eqref{a} in $\B$ with
\begin{equation*}
||y(t)-x(t)||\leq K_{f,\phi}\Phi(t);\qquad t\in{J}.
\end{equation*}
\end{defn}
Now we see Ulam-type stability of problem \eqref{a} by using successive approximations.
\begin{thm}
Suppose that $f$ satisfies assumption \text{(H1)}. For every $\epsilon>0,$ if $y:{J}\to{\R}^{n}$ in $\B$ satisfies inequality \eqref{j}, then there exists a unique solution $x:{J}\to{\R}^{n}$ in $\B$ of problem \eqref{a} with $x^{(k)}(1)=y^{(k)}(1),$ for $k=0,1,\cdots,m-1.$ Moreover, the problem \eqref{a} is Ulam-Hyers stable with
\begin{equation*}
||y(t)-x(t)||\leq\bigg(\frac{E_{\alpha}\big(\theta(\log{T})^{\alpha}\big)-1}{\theta}\bigg)\epsilon,\quad t\in{J},
\end{equation*}
and $\theta=\big(\frac{M}{1-N}\big)>0.$
\end{thm}
\begin{proof}
For every $\epsilon>0,$ let $y:{J}\to{\R}^{n}$ in $\B$ satisfies inequality \eqref{j}, then there exists a function $\sigma_y(t)\in\B$ (depending on $y$) such that
\begin{equation*}
||\sigma_y(t)||\leq\epsilon,\quad\text{and}\quad \D_{1}^{\alpha}y(t)=f(t,y(t),{\D_{1}^{\alpha}y(t)})+\sigma_y(t),\quad t\in{J}.
\end{equation*}
In the light of Lemma 6, $y$ satisfies the fractional integral equation
\begin{equation*}
y(t)=\sum_{k=0}^{m-1}\frac{y^{(k)}(1)}{\Gamma(k+1)}(\log{t})^k+\int_{1}^{t}\bigg(\log{\frac{t}{s}}\bigg)^{\alpha-1}\frac{p^{0}}{\Gamma(\alpha)}(s)\frac{ds}{s}
+\int_{1}^{t}\bigg(\log{\frac{t}{s}}\bigg)^{\alpha-1}\frac{\sigma_y(s)}{\Gamma(\alpha)}\frac{ds}{s},\,\, t\in{J},
\end{equation*}
where $p^{0}\in\B$ satisfies the functional equation $p^{0}(t)=f(t,y(t),p^{0}(t))$ for $t\in{J}.$

Define $x^{0}(t)=y(t)$ for $t\in{J}$ and consider the sequence $\{x^{j}\}\subseteq\B$ defined by
\begin{equation}\label{m}
x^{j}(t)=\sum_{k=0}^{m-1}\frac{y^{(k)}(1)}{\Gamma(k+1)}(\log{t})^k+\frac{1}{\Gamma(\alpha)}\int_{1}^{t}\bigg(\log{\frac{t}{s}}\bigg)^{\alpha-1}p^{j-1}(s)\frac{ds}{s},\quad t\in{J},
\end{equation}
where $p^{j-1}(t)\in\B \, (j\in\N)$ is such that
\begin{equation}\label{n}
p^{j-1}(t)=f(t,x^{j-1}(t),p^{j-1}(t)),\quad t\in{J}.
\end{equation}
By using the principle of mathematical induction, we prove that
\begin{equation}\label{o}
||x^{j}(t)-x^{j-1}(t)||\leq\frac{\epsilon}{\theta}\frac{[\theta(\log{t})^{\alpha}]^{j}}{\Gamma(\alpha j+1)},\quad j\in\N,t\in{J}.
\end{equation}
First we show that inequality \eqref{o} is true for $j=1.$ By definition of successive approximations, for any $t\in{J}$ we obtain
\begin{align*}
||x^{1}(t)-x^{0}(t)||&=\bigg{\|}\sum_{k=0}^{m-1}\frac{y^{(k)}(1)}{\Gamma(k+1)}(\log{t})^k+\frac{1}{\Gamma(\alpha)}\int_{1}^{t}\bigg(\log{\frac{t}{s}}\bigg)^{\alpha-1}p^{0}(s)\frac{ds}{s}-y(t)\bigg{\|}\\
&=\bigg{\|}\sum_{k=0}^{m-1}\frac{y^{(k)}(1)}{\Gamma(k+1)}(\log{t})^k+\frac{1}{\Gamma(\alpha)}\int_{1}^{t}\bigg(\log{\frac{t}{s}}\bigg)^{\alpha-1}p^{0}(s)\frac{ds}{s}\\
&\hspace{1cm}-\bigg(\sum_{k=0}^{m-1}\frac{y^{(k)}(1)}{\Gamma(k+1)}(\log{t})^k+\I_{1}^{\alpha}p^{0}(t)-\I_{1}^{\alpha}\sigma_y(t)\bigg)\bigg{\|}\\
&=||\I_{1}^{\alpha}\sigma_y(t)||\\
&\leq\frac{1}{\Gamma(\alpha)}\int_{1}^{t}\bigg(\log{\frac{t}{s}}\bigg)^{\alpha-1}||\sigma_y(s)||\frac{ds}{s}\\
&\leq\epsilon\frac{(\log{t})^{\alpha}}{\Gamma(\alpha+1)},\quad t\in{J},
\end{align*}
which proves inequality \eqref{o} for $j=1.$ Now, we assume that the inequality \eqref{o} hold for $j=r,r\in\N$ and prove it for $j=r+1.$ Again by definition of successive approximations, for any $t\in{J},$ we have
\begin{equation}\label{p}
||x^{r+1}(t)-x^{r}(t)||\leq\frac{1}{\Gamma(\alpha)}\int_{1}^{t}\bigg(\log{\frac{t}{s}}\bigg)^{\alpha-1}||p^r(s)-p^{r-1}(s)||\frac{ds}{s}.
\end{equation}
Since $p^{j}(t)=f(t,x^{j}(t),p^{j}(t)),\,\, t\in{J}$ and using assumption \text{(H1)}, we have
\begin{align*}
||p^{r}(t)-p^{r-1}(t)||&=||f(t,x^{r}(t),p^{r}(t))-f(t,x^{r-1}(t),p^{r-1}(t))||\\
&\leq M ||x^{r}(t)-x^{r-1}(t)||+N||p^{r}(t)-p^{r-1}(t)||\\
&=\theta||x^{r}(t)-x^{r-1}(t)||,\quad t\in{J}.
\end{align*}
Using the above estimate in inequality \eqref{p}, we obtain
\begin{align*}
||x^{r+1}(t)-x^{r}(t)||&\leq\frac{\theta}{\Gamma(\alpha)}\int_{1}^{t}\bigg(\log{\frac{t}{s}}\bigg)^{\alpha-1}[||x^r(s)-x^{r-1}(s)||]\frac{ds}{s}\\
&\leq\frac{\theta}{\Gamma(\alpha)}\int_{1}^{t}\bigg(\log{\frac{t}{s}}\bigg)^{\alpha-1}\bigg[\frac{\epsilon}{\theta}\frac{[\theta(\log{s})^{\alpha}]^{r}}{\Gamma(r\alpha+1)}\bigg]\frac{ds}{s}\\
&=\frac{\epsilon\theta^r}{\Gamma(r\alpha+1)}\bigg(\int_{1}^{t}\bigg(\log{\frac{t}{s}}\bigg)^{\alpha-1}\frac{(\log{s})^{r\alpha}}{\Gamma(\alpha)}\frac{ds}{s}\bigg)\\
&=\frac{\epsilon}{\theta}\frac{(\theta(\log{t})^\alpha)^{(r+1)}}{\Gamma(\alpha(r+1)+1)},\quad t\in{J},
\end{align*}
which is inequality \eqref{o} for $j=r+1.$ The proof of inequality \eqref{o} is completed by the principle of mathematical induction.

Furthermore, for any $t\in{J},$ from inequality \eqref{o}, we obtain
\begin{equation*}
||x^{j}(t)-x^{j-1}(t)||\leq\frac{\epsilon}{\theta}\sum_{j=1}^{\infty}\frac{(\theta(\log{T})^{\alpha})^j}{\Gamma(j\alpha+1)} \quad\text{and}\,\,j\in\N.
\end{equation*}
This gives
\begin{equation}\label{q}
||x^{j}(t)-x^{j-1}(t)||\leq\frac{\epsilon}{\theta}\big(E_\alpha(\theta(\log{T})^{\alpha})-1\big).
\end{equation}
Hence the series $x^{0}(t)+\sum_{j=1}^{\infty}[x^j(t)-x^{j-1}(t)]$ converges absolutely and uniformly on ${J}$ with respect to the norm $||\cdot||.$ Consider
\begin{equation}\label{r}
x(t)=x^{0}(t)+\sum_{j=1}^{\infty}[x^j(t)-x^{j-1}(t)],\quad t\in{J}.
\end{equation}
Then
\begin{equation*}
x^r(t)=x^{0}(t)+\sum_{j=1}^{r}[x^j(t)-x^{j-1}(t)]
\end{equation*}
is the $r^{th}$ partial sum of the series \eqref{r}, and gives
\begin{equation}\label{s}
\lim_{r\to\infty}||x^r(t)-x(t)||=0,\quad \text{for\,\,all}\,\,t\in{J}.
\end{equation}
Since convergence is uniform, $x\in\B.$ We prove that the limit function $x$ is a solution of
\begin{equation*}
x(t)=\sum_{k=0}^{m-1}\frac{y^{(k)}(1)}{\Gamma(k+1)}(\log{t})^k+\frac{1}{\Gamma(\alpha)}\int_{1}^{t}\bigg(\log{\frac{t}{s}}\bigg)^{\alpha-1}p(s)\frac{ds}{s},\quad t\in{J},
\end{equation*}
where $p\in\B$ satisfies the functional equation $p(t)=f(t,x(t),p(t)),\,\, t\in{J}.$

For any $t\in{J},$ we prove that $p^r\in\B,(r=0,1,\cdots)$ generated in \eqref{n} satisfies
\begin{equation}\label{t}
\lim_{r\to\infty}||p^r(t)-p(t)||=0.
\end{equation}
Using assumption $\text{(H1)},$ we obtain
\begin{align}\label{u}
||p^r(t)-p(t)||&=||f(t,x^r(t),p^r(t))-f(t,x(t),p(t))||\nonumber\\
&\leq M||x^r(t)-x(t)||+N||p^r(t)-p(t)||\nonumber\\
&=\theta||x^r(t)-x(t)||,\quad t\in{J}.
\end{align}
Further, using equation \eqref{s} in \eqref{u}, equation \eqref{t} can be easily proved. Again, by definition of successive approximations
\begin{align*}
\bigg{\|}x(t)-&\sum_{k=0}^{m-1}\frac{x_k}{\Gamma(k+1)}(\log{t})^k+\frac{1}{\Gamma(\alpha)}\int_{1}^{t}\bigg(\log{\frac{t}{s}}\bigg)^{\alpha-1}p(s)\frac{ds}{s}\bigg{\|}\\
&=\bigg{\|}x(t)-x^j(t)+\frac{1}{\Gamma(\alpha)}\int_{1}^{t}\bigg(\log{\frac{t}{s}}\bigg)^{\alpha-1}p^{j-1}(s)\frac{ds}{s}-\frac{1}{\Gamma(\alpha)}\int_{1}^{t}\bigg(\log{\frac{t}{s}}\bigg)^{\alpha-1}p(s)\frac{ds}{s}\bigg{\|}\\
&\leq||x(t)-x^j(t)||+\frac{1}{\Gamma(\alpha)}\int_{1}^{t}\bigg(\log{\frac{t}{s}}\bigg)^{\alpha-1}||p^{j-1}(s)-p(s)||\frac{ds}{s}.
\end{align*}
Note that left hand side of above inequality is independent of $j,$ taking limit as $j\to\infty,$ we obtain
\begin{equation}\label{v}
x(t)=\sum_{k=0}^{m-1}\frac{y^{(k)}(1)}{\Gamma(k+1)}(\log{t})^k+\frac{1}{\Gamma(\alpha)}\int_{1}^{t}\bigg(\log{\frac{t}{s}}\bigg)^{\alpha-1}p(s)\frac{ds}{s},\quad t\in{J}.
\end{equation}
This means $x(t)$ is solution of problem \eqref{a} with initial condition
\begin{equation*}
x^{(k)}(1)=y^{(k)}(1),\quad x^{(k)}(1),y^{(k)}(1)\in{\R}^{n},k=0,1,\cdots,m-1.
\end{equation*}
Lastly, from inequality \eqref{q} with series \eqref{r}, it follows that problem \eqref{a} is Ulam-Hyers stable with
\begin{equation}\label{guh}
||y(t)-x(t)||\leq\bigg(\frac{E_\alpha(\theta(\log{T})^{\alpha})-1}{\theta}\bigg)\epsilon,\quad t\in{J}.
\end{equation}

To prove uniqueness of solution $x(t),$ assume that $\bar{x}(t)$ is another solution of problem \eqref{a} with initial condition $\bar{x}^{(k)}(1)=y^{(k)}(1),\,\, x^{(k)}(1),y^{(k)}(1)\in{\R}^{n},k=0,1,\cdots,m-1.$ Then
\begin{equation*}
\bar{x}(t)=\sum_{k=0}^{m-1}\frac{y^{(k)}(1)}{\Gamma(k+1)}(\log{t})^k+\frac{1}{\Gamma(\alpha)}\int_{1}^{t}\bigg(\log{\frac{t}{s}}\bigg)^{\alpha-1}\bar{p}(s)\frac{ds}{s},\quad t\in{J},
\end{equation*}
where $\bar{p}\in\B$ satisfies $\bar{p}(t)=f(t,\bar{x}(t),\bar{p}(t)).$ Therefore
\begin{equation*}
||x(t)-\bar{x}(t)||\leq\frac{1}{\Gamma(\alpha)}\int_{1}^{t}\bigg(\log{\frac{t}{s}}\bigg)^{\alpha-1}||p(s)-\bar{p}(s)||\frac{ds}{s},\quad t\in{J}.
\end{equation*}
By hypothesis \text{(H1)},
\begin{equation*}
||p(t)-\bar{p}(t)||\leq\theta||x(t)-\bar{x}(t)||.
\end{equation*}
Hence
\begin{equation*}
||x(t)-\bar{x}(t)||\leq\frac{\theta}{\Gamma(\alpha)}\int_{1}^{t}\bigg(\log{\frac{t}{s}}\bigg)^{\alpha-1}||x(s)-\bar{x}(s)||\frac{ds}{s},\quad t\in{J}.
\end{equation*}
Applying Lemma 4 to above inequality with $u(t)=||x(t)-\bar{x}(t)||$ and $a(t)=0,$ we obtain $||x(t)-\bar{x}(t)||=0,$
for all $t\in{J}.$ The proof is completed.
\end{proof}
\begin{cor}
Suppose that all the assumptions of Theorem 2 are satisfied. Then the problem \eqref{a} is generalized Ulam-Hyers stable.
\end{cor}
\begin{proof}
Let $\psi(\epsilon)=\big(\frac{E_\alpha(\theta(\log{T})^{\alpha})-1}{\theta}\big)\epsilon$ in \eqref{guh} then $\psi(0)=0$ and problem \eqref{a} is generalized Ulam-Hyers stable.
\end{proof}
\begin{thm}
Suppose that assumptions \text{(H1)} and \text{(H2)} are satisfied. Then for every $\epsilon>0$ and $y:{J}\to{\R}^n$ in $\B$ satisfying inequality \eqref{l}, there exists a unique solution $x:{J}\to{\R}^{n}$ in $\B$ of problem \eqref{a} with $x^{(k)}(1)=y^{(k)}(1),\,\, k=0,1,\cdots,m-1,$ that satisfies
\begin{equation*}
||y(t)-x(t)||\leq\epsilon\bigg(\frac{K}{1-K\theta}\bigg)\Phi(t),\quad t\in{J}.
\end{equation*}
\end{thm}
\begin{proof}
For every $\epsilon>0,$ let $y:{J}\to{\R}^n$ in $\B$ satisfies inequality \eqref{l}. Then there exists a function $\sigma_y\in\B$ (depending on $y$) such that
\begin{equation*}
||\sigma_y(t)||\leq\epsilon\Phi(t),\quad \text{and}\quad \D_{1}^{\alpha}y(t)=f(t,y(t),{\D_{1}^{\alpha}y(t)})+\sigma_y(t),\quad t\in{J}.
\end{equation*}
By Lemma 6, $y$ satisfies the fractional integral equation
\begin{equation*}
y(t)=\sum_{k=0}^{m-1}\frac{y^{(k)}(1)}{\Gamma(k+1)}(\log{t})^k+\int_{1}^{t}\bigg(\log{\frac{t}{s}}\bigg)^{\alpha-1}\frac{p^{0}(s)}{\Gamma(\alpha)}\frac{ds}{s}+\int_{1}^{t}\bigg(\log{\frac{t}{s}}\bigg)^{\alpha-1}\frac{\sigma_y(s)}{\Gamma(\alpha)}\frac{ds}{s},\,\, t\in{J},
\end{equation*}
where $p^0\in\B$ satisfies the functional equation $p^(t)=f(t,y(t),p^0(t)),$ for $t\in{J}.$

Consider the sequence $\{x^{j}\}\subseteq\B$ defined by \eqref{m} with $x^0(t)=y(t),t\in{J}.$ By the principle of mathematical induction, we prove that
\begin{equation}\label{x}
||x^j(t)-x^{j-1}(t)||\leq\frac{\epsilon}{\theta}(K\theta)^{j}\Phi(t),\quad j\in\N,t\in{J}.
\end{equation}
First we show the inequality \eqref{x} is true for $j=1.$ For any $t\in{J},$ using definition of successive approximations and assumption \text{(H2)}, we have
\begin{align*}
||x^{1}(t)-x^{0}(t)||&=||x^{1}(t)-y(t)||\\
&=||\I_{1}^{\alpha}\sigma_y(t)||\\
&\leq\frac{1}{\Gamma(\alpha)}\int_{1}^{t}\bigg(\log{\frac{t}{s}}\bigg)^{\alpha-1}||\sigma_y(s)||\frac{ds}{s}\\
&\leq\frac{\epsilon}{\Gamma(\alpha)}\int_{1}^{t}\bigg(\log{\frac{t}{s}}\bigg)^{\alpha-1}\Phi(s)\frac{ds}{s}\\
&=\epsilon\bigg{\|}\frac{1}{\Gamma(\alpha)}\int_{1}^{t}\bigg(\log{\frac{t}{s}}\bigg)^{\alpha-1}\Phi(s)\frac{ds}{s}\bigg{\|}\\
&\leq\frac{\epsilon}{\theta} (K\theta)\Phi(t),\quad t\in{J}.
\end{align*}
Thus the inequality \eqref{x} holds for $j=1.$ Assume that inequality \eqref{x} is true for $j=r,r\in\N$ and using similar arguments as we presented in the proof of Theorem 2, we have
\begin{align*}
||x^{r+1}(t)-x^{r}(t)||&\leq\frac{\theta}{\Gamma(\alpha)}\int_{1}^{t}\bigg(\log{\frac{t}{s}}\bigg)^{\alpha-1}||x^{r}(s)-x^{r-1}(s)||\frac{ds}{s}\\
&\leq\frac{\epsilon}{\Gamma(\alpha)}(K\theta)^r\int_{1}^{t}\bigg(\log{\frac{t}{s}}\bigg)^{\alpha-1}\Phi(s)\frac{ds}{s}\\
&=\epsilon(K\theta)^r\bigg{\|}\frac{1}{\Gamma(\alpha)}\int_{1}^{t}\bigg(\log{\frac{t}{s}}\bigg)^{\alpha-1}\Phi(s)\frac{ds}{s}\bigg{\|}\\
&\leq\epsilon(K\theta)^rK\Phi(t).
\end{align*}
Therefore
\begin{equation*}
||x^{r+1}(t)-x^{r}(t)||\leq\frac{\epsilon}{\theta}(K\theta)^{r+1}\Phi(t),\quad t\in{J},
\end{equation*}
which is the inequality \eqref{x} for $j=r+1.$ By the principle of mathematical induction, inequality \eqref{x} is true for all $j$ and the proof of inequality \eqref{x} is completed. Now using inequality \eqref{x} and assumption $0<K\theta<1,$ we have
\begin{equation*}
\sum_{j=1}^{\infty}||x^{j}(t)-x^{j-1}(t)||\leq\frac{\epsilon}{\theta}\bigg(\sum_{j=1}^{\infty}(\theta K)^{j}\bigg)\Phi(t)=\frac{\epsilon}{\theta}\bigg(\sum_{j=0}^{\infty}(\theta K)^{j}-1\bigg)\Phi(t).
\end{equation*}
Therefore
\begin{equation}\label{y}
\sum_{j=1}^{\infty}||x^{j}(t)-x^{j-1}(t)||\leq\frac{\epsilon}{\theta}\bigg(\frac{1}{1-K\theta}-1\bigg)\Phi(t)=\epsilon\bigg(\frac{K}{1-K\theta}\bigg)\Phi(t).
\end{equation}
Since $\Phi(t)$ is continuous on compact set ${J},$ it is bounded. Clearly, from above inequality \eqref{y}, it follows that the series $x^{0}(t)+\sum_{j=1}^{\infty}[x^j(t)-x^{j-1}(t)]$ converges absolutely and uniformly on ${J},$ with respect to the norm $||\cdot||.$ Define
\begin{equation}\label{z}
x(t)=x^{0}(t)+\sum_{j=1}^{\infty}[x^j(t)-x^{j-1}(t)],\quad t\in{J},
\end{equation}
and following the proof of Theorem 2, finally we obtain
\begin{equation*}
||y(t)-x(t)||\leq\epsilon\bigg(\frac{K}{1-K\theta}\bigg)\Phi(t),\quad t\in{J}.
\end{equation*}
\end{proof}
\begin{cor}
Under the hypothesis of Theorem 3, the problem \eqref{a} is generalized Ulam-Hyers-Rassias stable with respect to $\Phi\in{C}(J,{\R}_+).$
\end{cor}
\begin{proof}
Set $\epsilon=1$ and $K_{f,\phi}=\frac{K}{1-K\theta},$ it directly follows that the problem \eqref{a} is generalized Ulam-Hyers-Rassias stable.
\end{proof}
\subsection{$E_\alpha$-Ulam stability}
Now we discuss the $E_\alpha$-Ulam-type stability of problem \eqref{a}. For this, we introduce the following definitions form \cite{wl}.
\begin{defn}
The problem \eqref{a} is $E_\alpha$-Ulam-Hyers stable if there exists a real number $K_f>0$ such that for each $\epsilon>0$ and each $y:{J}\to{\R}^{n}$ in $\B$ satisfies inequality \eqref{j}, then there exists a solution $x:{J}\to{\R}^{n}$ of problem \eqref{a} in $\B$ with
\begin{equation*}
||y(t)-x(t)||\leq K_fE_\alpha(\psi_f(\log{t})^{\alpha})\epsilon,\quad \psi_f\geq0,t\in{J}.
\end{equation*}
\end{defn}
\begin{defn}
The problem \eqref{a} is generalized $E_\alpha$-Ulam-Hyers stable if there exists a function $\Psi\in{C}({\R}_+,{\R}_+),\Psi(0)=0,$ such that for each $\epsilon>0$ and each $y:{J}\to{\R}^{n}$ in $\B$ satisfies inequality \eqref{j}, then there exists a solution $x:{J}\to{\R}^{n}$ of problem \eqref{a} in $\B$ with
\begin{equation*}
||y(t)-x(t)||\leq \Psi(\epsilon)E_\alpha(\psi_f(\log{t})^{\alpha})\epsilon,\quad \psi_f\geq0,t\in{J}.
\end{equation*}
\end{defn}
\begin{defn}
The problem \eqref{a} is $E_\alpha$-Ulam-Hyers-Rassias stable with respect to $\Phi\in{C}({J},{\R}_+)$ if there exists a real number $K_\phi>0$ such that for each $\epsilon>0,$ if $y:{J}\to{\R}^{n}$ in $\B$ satisfies inequality \eqref{l}, then there exists a solution $x:{J}\to{\R}^{n}$ of problem \eqref{a} in $\B$ with
\begin{equation*}
||y(t)-x(t)||\leq K_\phi E_\alpha(\psi_f(\log{t})^{\alpha})\Phi(t)\epsilon,\quad t\in{J}.
\end{equation*}
\end{defn}
\begin{defn}
The problem \eqref{a} is generalized $E_\alpha$-Ulam-Hyers-Rassias stable with respect to $\Phi\in{C}({J},{\R}_+)$ if there exists a real number $K_\phi>0$ such that if $y:{J}\to{\R}^{n}$ in $\B$ satisfies inequality \eqref{l}, then there exists a solution $x:{J}\to{\R}^{n}$ of problem \eqref{a} in $\B$ with
\begin{equation*}
||y(t)-x(t)||\leq K_\phi E_\alpha(\psi_f(\log{t})^{\alpha})\Phi(t),\quad t\in{J}.
\end{equation*}
\end{defn}
Now we see the first $E_{\alpha}-$Ulam stability result of problem \eqref{a} as follows.
\begin{thm}
Suppose that \text{(H1)} is satisfied. For every $\epsilon>0,$ if $y:{J}\to{\R}^{n}$ in $\B$ satisfies inequality \eqref{j}, then there exists a unique solution $x:{J}\to{\R}^{n}$ in $\B$ of problem \eqref{a} with $x^{(k)}(1)=y^{(k)}(1), k=0,1,\cdots,m-1,$ that satisfies
\begin{equation*}
||y(t)-x(t)||\leq\frac{1}{\theta}E_\alpha(\theta(\log{t})^{\alpha})\epsilon,\quad t\in{J}.
\end{equation*}
\end{thm}
\begin{proof}
Noting that $x^{0}(t)=y(t),$ and from inequation \eqref{o} and equation \eqref{r}, we can write
\begin{align*}
||y(t)-x(t)||\leq\sum_{j=1}^{\infty}||x^{j}(t)-x^{j-1}(t)||\leq\frac{\epsilon}{\theta}\sum_{j=0}^{\infty}\frac{[\theta(\log{t})^{\alpha}]^{j}}{\Gamma(j\alpha+1)}=\frac{\epsilon}{\theta}E_\alpha(\theta(\log{t})^{\alpha}),\qquad t\in{J}.
\end{align*}
This means the problem \eqref{a} is $E_\alpha-$Ulam-Hyers stable.
\end{proof}
\begin{re}
Set $\Psi(\epsilon)=\frac{\epsilon}{\theta}$ then $\Psi(0)=0,$ this proves the problem \eqref{a} is generalized $E_{\alpha}-$Ulam-Hyers-Rassias stable.
\end{re}
\begin{thm}
Suppose that assumptions \text{(H1)} and \text{(H2)} are satisfied. Then for every $\epsilon>0$ and $y:{J}\to{\R}^{n}$ in $\B$ satisfying the inequality \eqref{l}, there exists a unique solution $x:{J}\to{\R}^{n}$ in $\B$ of problem \eqref{a} with $x^{(k)}(1)=y^{(k)}(1), k=0,1,\cdots,m-1,$ that satisfies
\begin{equation*}
||y(t)-x(t)||\leq\epsilon\bigg(\frac{K^2+K}{1-K\theta}\bigg)\Phi(t)E_{\alpha}((\log{t})^{\alpha}) ,\quad t\in{J}.
\end{equation*}
\end{thm}
\begin{proof}
Noting that $x^{0}(t)=y(t),$ then from inequation \eqref{x} and equation \eqref{z}, we have
\begin{align}\label{aa}
||y(t)-x(t)||\leq\sum_{j=1}^{\infty}||x^{j}(t)-x^{j-1}(t)||\leq\frac{\epsilon}{\theta}\sum_{j=1}^{\infty}(K\theta)^{j}\Phi(t)=\epsilon\bigg(\frac{K}{1-K\theta}\bigg)\Phi(t),\,\, t\in{J}.
\end{align}
Applying $\I_{1}^{\alpha}$ on both sides of above inequality \eqref{aa} and using assumption \text{(H2)}, we get
\begin{align}\label{ab}
\I_{1}^{\alpha}||y(t)-x(t)||\leq\epsilon\bigg(\frac{K}{1-K\theta}\bigg)\I_{1}^{\alpha}\Phi(t)\leq\epsilon\bigg(\frac{K}{1-K\theta}\bigg)K\Phi(t),\quad t\in{J}.
\end{align}
By adding inequations \eqref{aa} and \eqref{ab}, we obtain
\begin{align*}
||y(t)-x(t)||+\I_{1}^{\alpha}||y(t)-x(t)||&\leq\epsilon\bigg(\frac{K}{1-K\theta}\bigg)\Phi(t)+\epsilon\bigg(\frac{K}{1-K\theta}\bigg)K\Phi(t)\\
&\leq \epsilon\bigg(\frac{K^2+K}{1-K\theta}\bigg)\Phi(t),\quad t\in{J}.
\end{align*}
Hence
\begin{equation}\label{ac}
||y(t)-x(t)||\leq\epsilon\bigg(\frac{K^2+K}{1-K\theta}\bigg)\Phi(t)-\I_{1}^{\alpha}||y(t)-x(t)||,\quad t\in{J}.
\end{equation}
But $||y(t)-x(t)||\geq0,$ then the inequality \eqref{ac} can be written as
\begin{align*}
||y(t)-x(t)||&\leq\epsilon\bigg(\frac{K^2+K}{1-K\theta}\bigg)\Phi(t)+\I_{1}^{\alpha}||y(t)-x(t)||\\
&=\epsilon\bigg(\frac{K^2+K}{1-K\theta}\bigg)\Phi(t)+\frac{1}{\Gamma(\alpha)}\int_{1}^{t}\bigg(\log{\frac{t}{s}}\bigg)^{\alpha-1}||y(s)-x(s)||\frac{ds}{s},\quad t\in{J}.
\end{align*}
Applying Lemma 4 to above inequality with $u(t)=||y(t)-x(t)||,$ $\bar{a}(t)=\epsilon(\frac{K^2+K}{1-K\theta})\Phi(t)$ and $\bar{g}(t)=\frac{1}{\Gamma(\alpha)},$ we obtain
\begin{align*}
||y(t)-x(t)||\leq\epsilon\bigg(\frac{K^2+K}{1-K\theta}\bigg)\Phi(t)E_{\alpha}((\log{t})^{\alpha}),\quad t\in{J},
\end{align*}
which shows that problem \eqref{a} is $E_{\alpha}-$Ulam-Hyers-Rassias stable.
\end{proof}
\begin{re}
If we set $\epsilon=1,$ then by $K_{f,\phi}=\frac{K^2+K}{1-K\theta},$ it follows that IVP \eqref{a} is generalized $E_{\alpha}-$Ulam-Hyers-Rassias stable.
\end{re}
\section{An example}
Let ${\R}^{2}$ be the normed space with the norm
\begin{equation*}
||x||=|x_1|+|x_2|, \quad x=(x_1,x_2)\in{\R}^{2}.
\end{equation*}
Consider the following system of nonlinear implicit fractional initial value problem
\begin{equation}\label{e1}
\begin{cases}
&\D_{1}^{\frac{5}{2}}x(t)=f(t,x(t),{\D_{1}^{\frac{5}{2}}x(t)}),\quad t\in[1,e],\\
&x^{(k)}(1)=x_k,\quad x_k\in{\R}^{2}, k=0,1,2,
\end{cases}
\end{equation}
where $x:[1,e]\to{\R}^{2}$ and $f:[1,e]\times{\R}^{2}\times{\R}^{2}\to{\R}^{2}$ is a nonlinear function defined by
\begin{align*}
f(t,x(t),{\D_{1}^{\frac{5}{2}}x(t)})&=f\big(t,(x_1(t),x_2(t)),({\D_{1}^{\frac{5}{2}}x_1(t)},{\D_{1}^{\frac{5}{2}}x_1(t)})\big)\\
&=\bigg(\frac{\log{(2+t)}}{1+|x_1(t)|+|x_2(t)|},\frac{|{\D_{1}^{\frac{5}{2}}x_1(t)}|+|{\D_{1}^{\frac{5}{2}}x_2(t)}|}{e^{t^2+1}(1+|{\D_{1}^{\frac{5}{2}}x_1(t)}|+|{\D_{1}^{\frac{5}{2}}x_2(t)}|)}\bigg),\quad t\in[1,e].
\end{align*}
For any $x=(x_1,x_2),y=(y_1,y_2),\bar{x}=(\bar{x}_1,\bar{x}_2),\bar{y}=(\bar{y}_1,\bar{y}_2)\in{\R}^2,$ we have
\begin{align*}
||f(t,x,y)-f(t,\bar{x},\bar{y})||&\leq||f(t,(x_1,x_2),(y_1,y_2),(\bar{y}_1,\bar{y}_2))||\\
&=\bigg{\|}\bigg(\frac{\log{(2+t)}}{1+|x_1|+|x_2|},\frac{|y_1|+|y_2|}{e^{t^2+1}(1+|y_1|+|y_2|)}\bigg)\\
&\hspace{2cm}-\bigg(\frac{\log{(2+t)}}{1+|\bar{x}_1|+|\bar{x}_2|},\frac{|\bar{y}_1|+|\bar{y}_2|}{e^{t^2+1}(1+|\bar{y}_1|+|\bar{y}_2|)}\bigg)\bigg{\|}\\
&=\bigg{\|}\bigg(\log{(2+t)}\bigg[\frac{1}{1+|x_1|+|x_2|}-\frac{1}{1+|\bar{x}_1|+|\bar{x}_2|}\bigg],\\
&\hspace{2cm}\frac{1}{e^{t^2+1}}\bigg[\frac{|y_1|+|y_2|}{1+|y_1|+|y_2|}-\frac{|\bar{y}_1|+|\bar{y}_2|}{1+|\bar{y}_1|+|\bar{y}_2|}\bigg]\bigg)\bigg{\|}\\
&=\bigg{\|}\bigg(\log{(2+t)}\bigg[\frac{|\bar{x}_1|-|x_1|+|\bar{x}_2|-|x_2|}{(1+|x_1|+|x_2|)(1+|\bar{x}_1|+|\bar{x}_2|)}\bigg],\\
&\hspace{2cm}\frac{1}{e^{t^2+1}}\bigg[\frac{|y_1|-|\bar{y}_1|+|y_2|-|\bar{y}_2|}{(1+|y_1|+|y_2|)(1+|\bar{y}_1|+|\bar{y}_2|)}\bigg]\bigg)\bigg{\|}\\
&=\log{(2+t)}\bigg|\frac{|\bar{x}_1|-|x_1|+|\bar{x}_2|-|x_2|}{(1+|x_1|+|x_2|)(1+|\bar{x}_1|+|\bar{x}_2|)}\bigg|\\
&\hspace{2cm}+\frac{1}{e^{t^2+1}}\bigg|\frac{|y_1|-|\bar{y}_1|+|y_2|-|\bar{y}_2|}{(1+|y_1|+|y_2|)(1+|\bar{y}_1|+|\bar{y}_2|)}\bigg|.
\end{align*}
For any $a,b\geq0,$ we have $1\leq(1+a+b).$ Therefore
\begin{align*}
||f(t,x,y)-f(t,\bar{x},\bar{y})||&\leq\log{(2+t)}|{(|\bar{x}_1|-|x_1|+|\bar{x}_2|-|x_2|)}|\\
&\hspace{2cm}+\frac{1}{e^{t^2+1}}|{(|y_1|-|\bar{y}_1|+|y_2|-|\bar{y}_2|)}|\\
&\leq\log{(2+t)}||(||\bar{x}||-||x||)||+\frac{1}{e^{t^2+1}}||({||y||-||\bar{y}||})||\\
&\leq\log{(2+e)}||\bar{x}-x||+\frac{1}{e^{2}}||y-\bar{y}||.
\end{align*}
Thus, function $f$ satisfies condition \text{(H1)} with $M=\log{(2+e)}>0$ and $0<N=\frac{1}{e^2}<1.$ Hence by Theorem 1, problem \eqref{e1} has a unique solution on $[1,e].$

Moreover, as shown in Theorem 2, for every $\epsilon>0$ if $y:[1,e]\to{\R}^2$ satisfies
\begin{equation}\label{e3}
||\D_{1}^{\frac{5}{2}}x(t)-f(t,x(t),{\D_{1}^{\frac{5}{2}}x(t)})||\leq\epsilon,\quad t\in[1,e],
\end{equation}
there exists a unique solution $x:[1,e]\to{\R}^2$ such that
\begin{equation*}
||y(t)-x(t)||\leq\bigg(\frac{E_{\frac{5}{2}}(\theta(\log{e})^{\frac{5}{2}})-1}{\theta}\bigg)\epsilon,\quad \text{for\,\,all}\,\,t\in[1,e],
\end{equation*}
where $\theta=\frac{M}{1-N}=\frac{\log{(2+e)}}{(1-\frac{1}{e^2})}=\frac{e^2\log{(2+e)}}{(e^2-1)}.$ Hence problem \eqref{e1} is Ulam-Hyers stable.

Next, by corollary 1, $\psi(\epsilon)=\frac{E_{\frac{5}{2}}(\theta)-1}{\theta}\epsilon$ then $\psi(0)=0$ which means the problem \eqref{e1} is generalized Ulam-Hyers stable.

Again, for every $\epsilon>0,$ if $y:[1,e]\to{\R}^{2}$ satisfies \eqref{e3}, by Theorem 4, there exists a unique solution $x:[1,e]\to{\R}^{2}$ such that
\begin{equation*}
||y(t)-x(t)||\leq\frac{1}{\theta}{E_{\frac{5}{2}}(\theta)},\quad \text{for\,\,all}\,\,t\in[1,e].
\end{equation*}
Thus IVP \eqref{e1} is $E_{\frac{5}{2}}-$Ulam-Hyers stable. By setting $\epsilon=1$ and using remark 2, the problem \eqref{e1} is generalized $E_{\frac{5}{2}}-$Ulam-Hyers stable.

Now define the function $\Phi(t)=CE_{\frac{5}{2}}[(\log{t})^{\frac{5}{2}}],$ where $C$ is constant. Then $\Phi(t)$ is nondecreasing function such that
\begin{align*}
\I_{1}^{\frac{5}{2}}(\Phi(t))=\I_{1}^{\frac{5}{2}}(CE_{\frac{5}{2}}(\log{t})^{\frac{5}{2}})=C\I_{1}^{\frac{5}{2}}\bigg(E_{\frac{5}{2}}(\log{t})^{\frac{5}{2}}\bigg)
\leq C E_{\frac{5}{2}}((\log{t})^{\frac{5}{2}})=\Phi(t),\quad t\in[1,e].
\end{align*}
Thus function $\Phi(t)$ satisfies the condition \text{(H2)} with $K=1.$ Further, $0<K\theta=\theta$ and $\theta=\frac{e^2(\log{(2+e)})}{(e^2-1)}=0.77924294258<1.$

For $\epsilon>0$ and sufficiently large value of $C,$ let $y:[1,e]\to{\R}^2$ satisfy
\begin{equation}\label{e4}
||\D_{1}^{\frac{5}{2}}y(t)-f(t,y(t),{\D_{1}^{\frac{5}{2}}y(t)})||\leq\epsilon\Phi(t),\quad t\in[1,e].
\end{equation}
Then by Theorem 3, we get a solution $x:[1,e]\to{\R}^{2}$ of equation \eqref{e1} satisfying
\begin{equation*}
||y(t)-x(t)||\leq\epsilon\frac{1}{(1-\theta)}\Phi(t),\quad t\in[1,e]
\end{equation*}
and hence the problem \eqref{e1} is $E_{\frac{5}{2}}$-Ulam-Hyers-Rassias stable.

Lastly, for every $\epsilon>0,$ if $y:[1,e]\to{\R}^{2}$ satisfies inequation \eqref{e3}, then by Theorem 5, there exists a solution $x:[1,e]\to{\R}^{2}$ such that
\begin{equation*}
||y(t)-x(t)||\leq\epsilon\frac{2}{(1-\theta)}\Phi(t)E_{\frac{5}{2}}((\log{t})^{\frac{5}{2}}),\quad t\in[1,e].
\end{equation*}
This means problem \eqref{e1} is $E_{\frac{5}{2}}-$Ulam-Hyers-Rassias stable and by remark 3, it is generalized $E_{\frac{5}{2}}-$Ulam-Hyers-Rassias stable.

\footnotesize{

}
\end{document}